\documentclass[a4paper,11pt]{article}
\usepackage{amsfonts}
\usepackage{latexsym}
\usepackage{amssymb}
\usepackage{amsmath}
\usepackage{times}
\usepackage{subfig}
\usepackage{amsthm}          
\usepackage{mathrsfs}        
\usepackage[dvips]{graphicx}
\usepackage[all]{xy}
\usepackage{bbm}

\usepackage{hyperref}

\usepackage{multirow}

\newtheorem{theo}{Theorem}

\newtheorem{coro}[theo]{Corollary}
\newtheorem{lemm}[theo]{Lemma}

\newtheorem*{theo*}{Theorem}

\theoremstyle{definition}

\newtheorem{rema}[theo]{Remark}

\newtheorem{exem}[theo]{Example}

\newcommand{\ovl}{\overline}

\newcommand{\mc}{\mathcal}

\newcommand{\ce}{\mathbb{C}}

\newcommand{\F}{\mathcal{F}}
\newcommand{\G}{\mathcal{G}}

\newcommand{\Z}{\mathbb{Z}}
\newcommand{\N}{\mathbb{N}}

\renewcommand{\P}{\mathbb{P}}

\newcommand{\C}{\mathcal{C}}
\newcommand{\Cs}{\mathscr{C}}

\newcommand{\Pc}{\mathcal{P}}
\newcommand{\Q}{\mathbb{Q}}

\newcommand{\ds}{\displaystyle}

\newcommand{\sii}{\Longleftrightarrow}

 \DeclareMathOperator{\End}{End}\DeclareMathOperator{\Hom}{Hom}

\DeclareMathOperator{\sing}{Sing}

\newcommand{\fol}{\mathop{\mathbb{F}ol}\nolimits}

\renewcommand{\tilde}{\widetilde}

\DeclareMathOperator{\fix}{Fix}

\title{DEGREE OF THE FIRST INTEGRAL OF A\\ FOLIATION IN THE PENCIL $\Pc_4$}
\author{{\bf Liliana PUCHURI MEDINA}\footnote{Instituto de Mat\'ematica Pura e Aplicada 
(IMPA), Estrada Dona
Castorina 110, Rio de Janeiro, RJ, CEP 22460-320, Brazil,
{\tt lilianap@impa.br}.
The work of this author was partially supported by CNPQ}}
\date{}
\begin{document}

\maketitle

\begin{abstract}
Let $\Pc_4$ be the linear family of foliations of degree $4$ in $\P^2$ given by A. Lins Neto, whose set of parameters with first integral $I_p(\Pc_4)$ is dense and countable.
In this work, we will calculate explicitly the degree of the rational first integral of the foliations in this linear family, as a function of the parameter.

\end{abstract}

\section{Introduction}
One of the main problems in the theory of planar vector fields is to characterize the 
ones which admit a first integral.
The invariant algebraic curves are a central object in integrability theory since 1878, 
year when Darboux found connections between algebraic curves and the existence of first integrals of polynomial vector fields.
Thus, the first question was to know if a polynomial vector field has or not invariant algebraic curves, which was partially answered by Darboux in~\cite{GD}.
The most important improvements of Darboux's results  were given by Poincar\'e in 1891, who tried to answer the following question:
\begin{quote}
``Is it possible to decide if a foliation in $\P^2$ has a rational first integral?''
\end{quote}
This problem is known as the \emph{Poincar\'e Problem}.
In~\cite{Poin}, he observed that it is sufficient to bound the degree
of a possible algebraic solution. By imposing conditions on the singularities of the foliation he
obtains necessary conditions which guarantee the existence of a rational first integral. 
More recently, this problem has been reformulated 
as follows: 
given a foliation on $\P^2$, try to  bound the degree of the generic solution using information depending only on the foliation, 
for example its degree or the eingenvalues of its singularities.

Several authors studied this problem, see for instance~\cite{Ca,LNCe,Z,So}. 
In 2002, Lins Neto (cf.~\cite{LN1}) built some notable 1-parameter families of foliations in
$\P^2$, where the 
set of parameters in which the foliation has a first integral
is dense
and countable. The importance of these families is that there is no bound depending only on the degree 
and the analytic type of their singularities.
One of such
families is the pencil $\Pc_4$ in $\P^2$, whose set of parameters of foliations which have a first integral, denoted by $I_p(\Pc_4)$,
is the imaginary quadratic field $\Q(\tau_0)$, where $\tau_0=e^{2\pi i/3}$.

The purpose of this work is to calculate the degree of the foliations in $\Pc_4$
with rational first integral as a function of the parameter. For this, we first relate the
pencil $\Pc_4$ with a pencil of linear foliations $\Pc_4^*$ in a complex
torus $E \times E$, where $E=\ce / \langle 1, \tau_0 \rangle$. Then we derive the formula of the degree using the ideal norm of the ring
$\Z[\tau_0]$ as sketched below. Consequently, we are capable to address the
Poincar\'e Problem for the foliations in $\Pc_4$.

Given a foliation $\F_t  \in \Pc_4$, with $t\in I_p(\Pc_4)$ there exists an
unique foliation
$\G_{\alpha(t)}\in \Pc_4^*$  where $\alpha(t)=\ds\frac{t-1}{-2-\tau_0}$. Then
writing
$\alpha(t)=\ds\frac{\alpha_1}{\beta_1}$, with
$\alpha_1,\beta_1\in\Z[\tau_0]$ and $(\alpha_1,\beta_1)=1$, we have proved the
following result:
\begin{theo*}
If  $d_t$ is the degree of the first integral of $\F_t^4$ then
\[
d_t=N(\beta_1)+N(\alpha_1)+N(\beta_1-\alpha_1)+N(\beta_1+\tau_0\alpha_1),
\]
where $N(\beta)=a^2+b^2-ab$, for $\beta=a+\tau_o b \in\Z[\tau_0]$.
\end{theo*}

Besides we compute the growth of the
function which associates to every $n\in \N$, the number of parameters for which
the corresponding foliation has a first integral of degree at most $n$.
More specifically, if $\pi_{\Pc_4}(n)$ denote the number of parameters with
first integral of degree at most $n\in \N$, then
\[
\ds \pi_{\Pc_4}(n)=O(n^2).
\]

\section{Preliminaries}
Let $K\subset \ce$ an \emph{algebraic number field}  and $\mathcal{O}_K$ the ring of algebraic integers contained in $K$.
Given an ideal $I$ of $\mathcal{O}_K$ we consider the quotient ring $\ds\mathcal{O}_K / I$ which is finite (cf.~\cite[p. 106]{Tall}). 
The~\emph{ideal norm} of $I$, denoted by $N_{\mathcal{O}_K}(I)$,  is the cardinality of the $\ds\mathcal{O}_K / I$.

The~\emph{Dedekind Zeta Function of K} is defined for a complex number $s$ with $Re(s)>1$, by the Dirichlet series
\[
\zeta_K(s)=\sum_{I\subset \mathcal{O}_K}\frac{1}{N_{\mathcal{O}_K}(I)^s},
\]
where $I$ ranges through the non-zero ideals of the ring of integers
$\mathcal{O}_K$ of $K$.
This sum converges absolutely for all complex numbers $s$ with $Re(s)
> 1$. Note that $\zeta_{\Q}$
coincides with the  Riemann zeta function.

Let $E=\ce/\Gamma$ be an elliptic curve, where $\Gamma=\langle1,\tau\rangle$
and $\End(E):=\Hom(E,E)$. Then the field $\End(E)\otimes \Q$ is isomorphic  to a
number field $K$ such that $\mathcal{O}_K\simeq\End(E)$. Let
$\alpha,\beta\in\End(E)$, then define the morphism $\varphi_{\alpha,\beta}:E\to E\times E$ as
\[
\varphi_{\alpha,\beta}(x)=(\alpha x,\beta x). 
\]
Note that the image $E_{\alpha,\beta}$ of $\varphi_{\alpha,\beta}$ is an
elliptic curve.
Given $\alpha,\beta,\gamma,\delta\in\End(E)$, then the \emph{ intersection
number}
of the elliptic curves $E_{\alpha,\beta}$ and $E_{\gamma,\delta}$ is given by
\begin{equation}\label{eq:eq1}
E_{\alpha,\beta}\cdot E_{\gamma,\delta}=\frac{N_{\mathfrak{O}_K}\left(\det\left(\begin{array}{cc}
\alpha&\beta\\
\gamma&\delta
\end{array}\right)\right)}{N_{\mathfrak{O}_K}(\alpha,\beta)N_{\mathfrak{O}_K}(\gamma,\delta)},
\end{equation}
where $N_{\mathcal{O}_K}(a_1,\ldots,a_r)$ is the norm of the ideal generated by
$a_1,\ldots,a_r\in\End(E)$ (cf.~\cite[Lemma 3]{Shim}).

 As an application consider the following example:
\begin{exem}\label{ex1}
Let the elliptic curve $E=\ce/\langle1,\tau_0\rangle$, with
$\tau_0=e^{2\pi i/3}$, then  $\End(E) \simeq \Z[\tau_0]$. 
Given $\alpha=a +\tau_0 b \in \Z[\tau_0]$  the norm of ideal $\langle \alpha\rangle$ 
is $N_{\Z[\tau_0]}(\alpha)=|\alpha|^2=a^2+b^2-ab$. By~\eqref{eq:eq1}, given
$\alpha,\beta,\gamma,\delta\in\Z[\tau_0]$ such that $(\alpha,\beta)=1$ and
$(\gamma,\delta)=1$ the intersection number
of the elliptic curves $E_{\alpha,\beta}$ and $E_{\gamma,\delta}$ is
\begin{equation}\label{ejex1}
E_{\alpha,\beta}\cdot E_{\gamma,\delta}=N_{\Z_{\tau_0}}( \alpha \gamma-\beta \delta).
\end{equation}
\end{exem}

From now on, $\tau_0$ will denote the complex number $e^{2\pi i/3}$.

\subsection{The pencil $\mathcal{P}_4$ in $\mathbb{P}^2$ and the configuration $\mathcal{C}$}
 In~\cite[\S 2.2]{LN1}, Lins Neto defines the pencil $\Pc_4=\{\F_{\alpha}^4\}_{\alpha\in\ovl{\ce}}$ of degree 4 in $\P^2$,
where $\F_{\alpha}^4$ is defined by the 1-form $\omega+\alpha\eta$, where
\[
\begin{aligned}
	\omega &= (x^3-1)xdy-(y^3-1)ydx,\\
	\eta &= (x^3-1)y^2dy-(y^3-1)x^2dx,
\end{aligned}.
\]
 Let us state some properties of the pencil $\Pc_4$:
\begin{enumerate}
\item The tangency set of the pencil $\Pc_4$, given by $\omega\wedge \eta=0$, is the algebraic curve
\[
\Delta(\Pc_4)=\big\{[x:y:z]
\in\P^2\::\:(x^3-z^3)(y^3-z^3)(x^3-y^3)=0 \big\}.
\]
Then $\Delta(\Pc_4)$ is formed by nine invariant lines. Besides, the set of intersections of these lines 
is formed by twelve points. We will denote such lines and points by $\mc{L}=\{L_1,\ldots,L_9\}$ and $P=\{e_1,\ldots,e_{12}\}$.

\item If $\alpha\notin\{1,\tau_0,\tau_0^2,\infty\}$ then $\F_{\alpha}$
 has 21 non-degenerated singularities, where nine of them are of type $(-3:1)$, 
and the remaining twelve are radial singularities contained in $P$. 
In particular, $\F_{\alpha}$ has degree 4.
\item If $\alpha\in\{1,\tau_0,\tau_0^2,\infty\}$ then
$\sing(\F_{\alpha})=P$.
\end{enumerate}
Let $\C=\{\mc{L},P\}$
be the configuration of points and
the
nine lines in $\P^2$, as showed in Figure~\ref{fig1}.
\begin{figure}[h]
\begin{center}
\input{fignn3.pspdftex}
\end{center}
\caption{}\label{fig1}
\end{figure}

\subsection{The pencil $\mathcal{P}_4^*$ }\label{secp4conj}
 Let $E=\ce/\Gamma$ be an elliptic curve,
 where $\Gamma=\langle 1,\tau\rangle$ and $X=E\times E$. Let $(x,y)$ be a system of coordinates of $\ce^2$ and $\pi:\ce^2\to X$ be the natural projection.
  Let
$\Pc_1=\{\F_{\alpha}\}_{\alpha\in\ovl{\ce}}$ be the pencil of linear foliations in $\ce^2$,
 where $\F_{\alpha}$ is induced  by the 1-form
\begin{equation}\label{eq:linears1}
\omega_{\alpha}=dy-\alpha dx.
\end{equation}
Then, using $\pi$, we obtain a pencil of linear foliations $\Pc=\{\G_{\alpha}\}_{\alpha\in\ovl{\ce}}$ in
$X$. Define
\[
I_p(\Pc):=\{\alpha\in\ovl{\ce}\::\:\G_{\alpha}\mbox{ has an holomorphic first integral}\}.
\]

Given $\alpha\in\ce\setminus\{0\}$, let $L_{\alpha}= \big\{ \big(\pi(x),\pi(\alpha x)\big)\::\:x\in \ce\big\}$ be
the leaf of $\G_{\alpha}$ passing though $(0,0)$. Then:
\begin{align*}
 \#\big(L_{\alpha}\cap(\{0\}\times E)\big)<\infty &\sii \exists\, k\in\N\::\:k\alpha(m+\tau n)\in\Gamma,\,\forall\,m,n\in\Z,\\
						    &\sii \exists\, k\in\N\::\:k\Gamma(\alpha)\subset\Gamma,\mbox{ where }\Gamma(\alpha)=\alpha\Gamma.
\end{align*}
In particular, for $\alpha\in\ce\setminus\{0\}$, $\G_{\alpha}$ has an holomorphic first integral if, 
and only if, there exists $k\in \N$ 
such that $k\Gamma(\alpha)\subset \Gamma$. So we have the following Lemma.
\begin{lemm}\label{lem:lem1}
Let $\Pc=\{\G_{\alpha}\}_{\alpha\in\ovl{\ce}}$ be a pencil of linear foliations in $X$, as above. Then
\[
I_p(\Pc)= (\Q+\tau \Q)\cup\{\infty\}.
\]
\end{lemm}
In the case $\Gamma_0=\langle 1,\tau_0 \rangle$ and $E_0=\ce/\Gamma_0$, denoted $X_0= E_0\times E_0$. The pencil
$\{\G_{\alpha}\}_{\alpha\in\ovl{\ce}}$ in $X_0$ induced by~\eqref{eq:linears1}
will be denoted by
$\Pc_4^*$.
In particular, by  Lemma~\ref{lem:lem1} we have
\begin{equation}
I_p(\Pc_4^*)=
(\Q+\tau_0 \Q)\cup\{\infty\}=\Q(\tau_0)\cup\{\infty\}.
\end{equation}
\subsection{The configuration $\mathcal{C}^*$ in $X_0$}

Let $\varphi:X_0\to X_0$ be the holomorphic map defined by $\varphi(x,y)= (\tau_0 x,\tau_0 y)$. 
Then,
\begin{enumerate}
\item $\varphi^3=id_X$.
\item Defining $p_1=0$, $p_2=\frac{2}{3}+\frac{1}{3}\tau_0$ and $p_3=\frac{1}{3}+\frac{2}{3}\tau_0$ then
 $\fix(\varphi)=\big\{(p_l,p_k) \big\}_{l,k=1}^3$ is the set of the nine fixed
points of $\varphi$. Denote by $\{l_k\}_{k=1}^9$ the nine fixed points of $\varphi$, then
\[
\fix(\varphi)=\{l_1,\ldots,l_{9}\}.
\]
\end{enumerate}

\noindent Now consider the four elliptic curves in $X_0$:
\begin{align*}
E_{0,1}&=\{0\}\times E_0,& E_{1,1}&=\big\{(x,x)\::\:x\in E_0\big\},\\
E_{1,0}&=E_0\times\{0\},& E_{1,-\tau_0}&=\big\{(x,-\tau_0 x)\::\:x\in E_0 \big\}.
\end{align*}
Let $\Cs$ the set of these four elliptic curves.  
Given $F\in\Cs$ and $p\in\fix(\varphi)$, denote $F_p=F+p$. 
Hence, the set $
\mc{E}:=\{F_p\::\: p\in\fix(\varphi), F\in\Cs\}$
consists of twelve elliptic curves, which we denote $E_1,\ldots, E_{12}$, that is,
\[
\mc{E}=\{E_1,\ldots,E_{12}\}.
\]
Since,
$\varphi(F_p)=F_p$ and $\fix(\varphi)\cap F_p=(\fix(\varphi)\cap F)+p$ then
fixed two different elliptic curves they 
intersect only in three fixed points of $\varphi$.

Let $\C^*=\big(\fix(\varphi),\mc{E}\big)$ be the configuration of points and
elliptic
curves in $X_0$, 
showed in Figure~\ref{fig2}.
\begin{figure}[h]
\begin{center}
\input{fignn4.pspdftex}
\end{center}
\caption{}\label{fig2}
\end{figure}

\section{Relation between the pencils $\mathcal{P}_4^*$ and $\mathcal{P}_4$ }
The relation between the pencils $\Pc_4^*$ and $\Pc_4$ was given by McQuillan in~\cite[p. 108]{BR1}, where 
he proved the existence of a rational map $g:X_0\dashrightarrow\P^2$ such that $g^*(\Pc_4)=\Pc_4^*$.
 We now give an idea of how the function $g$ is constructed. We refer the reader to~\cite{LN-Puch} for the details.

Let $\pi:Bl_{\fix(\varphi)}(X_0)\to X_0$ be obtained
 from $X_0$ by blowing-up the nine fixed points of $\varphi$, and denote $D_k=\pi^{-1}(l_k)$, for $k=1,\ldots,9$. 
Defining $\tilde{X}=Bl_{\fix(\varphi)}(X_0)$, there is an automorphism 
$\tilde{\varphi}:\tilde{X}\to\tilde{X}$
such that $\pi\circ\tilde{\varphi}=\varphi\circ\pi$.
Let $\tilde{Y}=\tilde{X}/\langle\tilde{\varphi}\rangle$ then $\tilde{Y}$ is a
smooth rational surface such that the quocient map $\tilde{h}:\tilde{X} \to
\tilde{Y}$
is a finite morphism with degree 3, and its ramification divisor is $R=\sum_{i=1}^9 3D_k$.

Since, $\tilde{h}|_{D_i}:D_i \to h(D_i)$ is a biholomorphism, the rational map $\tilde{h}$ maps $D_i$ in a rational curve with autointersection $-3$,
for $i=1,\ldots,9$. 
Besides $\tilde{h}$ maps each  elliptic curve $\pi^*E_i$, $E_i \in \mc{E}$, in a rational
curve $\tilde{E}_i$ with autointersection  $-1$, for $i=1,\ldots,12$,
as showed in Figure~\ref{fig3}.
\begin{figure}[h]
\begin{center}
\input{fignn5.pspdftex}
\end{center}
\caption{}\label{fig3}
\end{figure}

\subsection{Relation between $\mathcal{C}$ and $\mathcal{C}^*$}
Let $\pi_1:\tilde{Y}\to Y_0$ be the blowing-down map of the curves
$\tilde{E}_1,\ldots,\tilde{E}_{12}$.
\begin{lemm}
With the notations above defined we have that $Y_0=\P^2$.
\end{lemm}
\begin{proof}
By the Riemann-Hurwitz formula for surfaces 
we have
\[
c_2(\tilde{X})=3c_2(\tilde{Y})-\sum_{i=1}^9 2\chi(D_k),
\]
where $c_2(\tilde{X})=9$ and $\chi(D_k)=2$, for $k=1,\ldots, 9$. Therefore, $c_2(\tilde{Y})=15$ and $c_2(Y_0)=3$. 
This implies that $Y_0$ is a minimal surface, by the Noether formula (cf.~\cite{GH}). Since the only minimal
rational surfaces are $\P^2$ and the Hirzebruch surfaces $S_n$, with $n\geq 2$, we have $Y_0=\P^2$ because
$c_2(S_n)\geq 4$.
\end{proof}

Let the rational map
\[
g=\pi_1^{-1}\circ
\tilde{h}\circ\pi:X_0\dashrightarrow Y_0=\P^2
\]
(see the figure~\ref{fig4}). 
Let $\mc{E}_*:=g(\mc{E})$ and $\fix(\varphi)_*:=g(\fix(\varphi))$
Then $g$ maps each elliptic curve $E \in \mc{E}$ in
a point
in
$\P^2$, so $\mc{E}_*$ consist of twelve points  in $\P^2$. Besides, 
$g$ maps  each $l\in \fix(\varphi)$ in an algebraic curve $L
$ in $\P^2$ such that $L \cdot L=1$. In particular, $L$ is a line in
$\P^2$ and so, $\fix(\varphi)_*$ consist of nine lines. 
Besides the configuration $\{\mc{E}_*,\fix(\varphi)_*\}$ 
of points and  lines in $\P^2$ satisfy
the following properties  

\begin{enumerate}
\item Each line  in $ \fix(\varphi)_*$ contains four points of
$\mc{E}_*$.
\item Each point of $\mc{E}_*$ belongs to two lines of $
\fix(\varphi)_*$.
\item If three points of $\mc{E}_*$ are not in a line  in $
\fix(\varphi)_*$ then the points are not aligned.
\end{enumerate}
Then, by Proposition 1 of
 \cite{LN1}, unless an automorphism of $\P^2$, we can to suppose que
the configuration obtained is the configuration $\C$, that is, 
$\C=\big(\fix(\varphi)_*,\mc{E}_* \big)$ that has been described in
the section~\ref{secp4conj}.

\subsection{Relations between the foliations in $\mathcal{P}_4$ and $\mathcal{P}_4^*$}
Recall that, fixed $\alpha \in \ovl{\ce}$, the foliation $\G_{\alpha} \in \Pc_4^*$ in $X_0$ is induced by the 
$\omega_{\alpha}=dy-\alpha dx$. Since the 1-form $\omega_{\alpha}$ is $\varphi$-invariant, 
$\G_{\alpha}$ induces a foliation $\F_{\alpha}$ in $\P^2$ as showed in Figure~\ref{fig4}.
\begin{figure}[h]
\begin{center}
\input{figurabase6.pspdftex}
\end{center}
\caption{}\label{fig4}
\end{figure}
Besides, all the lines of $\fix(\varphi)_*$ are invariant respect to $\F_{\alpha}$. Then 
by (cf. \cite[\S 2.2]{LN1}) there exists an unique $\Lambda(\alpha)\in\ovl{\ce}$ such that 
$\F_{\alpha}=\F^4_{\Lambda(\alpha)}$, where $\F^4_{\Lambda(\alpha)} \in \Pc_4$. In particular $g^*(\Pc_4)=\Pc_4^*$.
\begin{lemm}
The rational function $\Lambda:\ovl{\ce} \to \ovl{\ce}$ is a M\"{o}bius map defined by $\Lambda(\alpha)=(\tau_0^2-1)\alpha+1$.
\end{lemm}
\begin{proof}
Since $\F_{\Lambda(0)}$, $\F_{\Lambda(1)}$, $\F_{\Lambda(-\tau_0)}$ and $\F_{\Lambda(\infty)}$ have 
twelve singularities, we have
\[
\big\{\Lambda(0),\Lambda(1),\Lambda(-\tau_0),\Lambda(\infty)\big\}=\{1,\tau_0,
\tau_0^2 , \infty\}.
\]
The configurations $\C^*$ in $X$ and $\C$ in $\P^2$ (see Figures~\ref{fig1} and \ref{fig2}), imply
\begin{align*}
g^*(\F_{\infty}^4)&=\G_{\infty}, &g^*(\F_1^4)&=\G_0,\\
g^*(\F_{\tau_0^2}^4)&=\G_1, &g^*(\F_{\tau_0}^4)&=\G_{-\tau_0}.
\end{align*}
Then $\Lambda:\ovl{\ce} \to \ovl{\ce}$ is an injective function such that
$\Lambda(\infty)=\infty$, $\Lambda(0)=1$, $\Lambda(1)=\tau_0^2$ and
$\Lambda(-\tau_0)=\tau_0$.
Therefore $\Lambda(\alpha)=(\tau_0^2-1)\alpha+1=(-2-\tau_0)\alpha+1$.
\end{proof}
\begin{rema}
If we have a automorphism of $\P^2$ preserving the configuration 
$\C=\big(\mc{E}_*,\fix(\varphi)_*\big)$ of points and 
lines, then $\Lambda$ is a M\"{o}bius map such that
\[
\big\{\Lambda(0),\Lambda(1),\Lambda(-\tau_0),\Lambda(\infty)\big\}=\{1,\tau_0,
\tau_0^2 , \infty\}.
\]
\end{rema}

\section{Calculation of the  degree of the first integral of a foliation $\mathcal{F}_t^4\in \mathcal{P}_4$ , $t\in \mathbb{Q}(\tau_0)$.}

Let $\F_t^4\in \Pc_4$, with $t\in \Q(\tau_0)$. Then there exists an unique foliation
$\G_{\alpha}\in \Pc_4^*$ such that $g^*(\G_{\alpha})=\F_t^4$, where $\alpha=\Lambda^{-1}(t)$. 
Since $\Z[\tau_0]$ is a unique factorization domain, we can choose $\alpha_1,\beta_1\in\Z[\tau_0]$ and $(\alpha_1,\beta_1)=1$, such that 
$\alpha=\ds\frac{\alpha_1}{\beta_1}$.
In particular,  $\G_{\alpha}$ is induced by the 1-form $\omega=\beta_1 dy-\alpha_1 dx$. Besides, 
$f_{\alpha_1,\beta_1}=\beta_1 y-\alpha_1 x$ is a first integral of
$\G_{\alpha}$ and 
\[
E_{\alpha_1,\beta_1}=\big\{(\alpha_1 x,\beta_1 x): x \in E \big\} 
\]
is  the leaf of $\G_{\alpha}$ passing by $(0,0)$.

Let $F_t$ be the rational first integral of $\F^4_t$ of degree $d_t$. We want to determine $d_t$. 
For this, let $C$ a generic irreducible fiber of $F_t$ of degree $d_t$. 
We can suppose that $C^*:=g^*(C)=E_{\alpha_1,\beta_1}+p$, where $p \notin
\fix(\varphi)$. 
Let $C_{1,0}^*:=E_{1,0}+p$ in $X_0$ and $C_{1,0}=g(C_{1,0}^*)$ the
curve obtained in $\P^2$. 
The idea for calculate $d_t$ is to find the relation between the intersection of $C$ and $C_{1,0}$ in $\P^2$ and the intersection
of $C^*$ and $C_{1,0}^*$ in $X_0$ (see Figure~\ref{fig5}). 
\begin{figure}[h]
\begin{center}
\input{figurabase13.pspdftex}
\end{center}
\caption{}\label{fig5}
\end{figure}

We observe that
\begin{equation}\label{eq:eqeq}
 d_t\deg(C_{1,0})=C\cdot C_{1,0}=\pi_1^*(C)\cdot \pi_1^*(C_{1,0}). 
\end{equation}
Since $C_{1,0}\cap L_7=\{e_{10},e_5,e_9\}$ (see Figure~\ref{fig1}), where $e_{10},e_5,e_9$ are radial singularities 
of $\F^4_1$ and $\pi_1^*C_{1,0}$ is a regular curve, we have  $\deg(C_{1,0})=3$.
Let $\tilde{C}$ and $\tilde{C}_{1,0}$ the strict transformations of $C$ and $C_{1,0}$ by $\pi_1$, respectively,
then
\begin{equation}\label{eq:eq5}
\pi_1^*(C)=\tilde{C}+\sum_{p\in \mc{E}_*\cap C}m_pD_p,
\end{equation}
where  $m_p$ is the multiplicity of $C$ in $p$ and $D_p=\pi_1^{-1}(p)$. Besides
\begin{equation}\label{eq:eq6}
\pi_1^*(C_{1,0})=\tilde{C}_{1,0}+\sum_{p\in \mc{E}_*\cap C_{1,0}}D_p,
\end{equation}
where $\mc{E}_*\cap C_{1,0}=\mc{E}_*\setminus\{e_1,e_6,e_8\}$.

Combining~\eqref{eq:eq5} and \eqref{eq:eq6} in~\eqref{eq:eqeq} we obtain
\begin{equation}\label{equation8}
\begin{split}
3 d_t
		  &=\tilde{C}\cdot \tilde{C}_{1,0}+\sum_{p\in \mc{E}_*\cap C_{1,0}}\tilde{C}\cdot D_p+
\sum_{p\in \mc{E}_*\cap C}m_p\tilde{C}_{1,0}\cdot D_p-
\sum_{p\in \mc{E}_*\cap C_{1,0}}m_p\\
		  &=\tilde{C}\cdot \tilde{C}_{1,0}+\sum_{p\in \mc{E}_*\cap C_{1,0}}\tilde{C}\cdot D_p+
\sum_{p\in \mc{E}_*\cap C_{1,0}}m_p\tilde{C}_{1,0}\cdot D_p-
\sum_{p\in \mc{E}_*\cap C_{1,0}}m_p\\
\end{split} 
\end{equation}
Now, given $p\in \mc{E}_*\cap C_{1,0}$ we have
\begin{align*}
\tilde{C}_{1,0}\cdot D_p&=C^*_{1,0}\cdot E_p=1,\\
 \tilde{C}\cdot D_p & =C^*\cdot E_p=m_p,
\end{align*}
where $E_p \in \mc{E}$ is a elliptic curve in $X_0$ such that
 $g(E_p)=p$. 
Hence in~\eqref{equation8},
\begin{equation}\label{eqqq}
 3 d_t=\tilde{C}\cdot \tilde{C}_{1,0}+\sum_{p\in \mc{E}_*\cap C_{1,0}} C^*\cdot E_p.
\end{equation}

Let $\tilde{C^*}_{1,0}$ and $\tilde{C^*}$, the strict 
transformations of $C_{1,0}^*$ and $C^*$ by $\pi$, respectively, then 
$C^*\cdot C^*_{1,0}=
\tilde{C^*}\cdot\tilde{C^*}_{1,0}.
$
Since $\tilde{h}^*\tilde{C}=3\tilde{C^*}$ and
 $\tilde{h}^*\tilde{C}_{1,0}=3\tilde{C^*}_{1,0}$ then by the Projection Formula, we have
\[
3C^*\cdot3 C^*_{1,0}=3\tilde{C^*}\cdot 3\tilde{C^*}_{1,0}=(\tilde{h}^*\tilde{C}\cdot \tilde{h}^*\tilde{C}_{1,0})=3\tilde{C}\cdot\tilde{C}_{1,0}, 
\]
obtaining $3 C^*\cdot C^*_{1,0}=\tilde{C}\cdot \tilde{C}_{1,0}$. Then, in~\eqref{eqqq}, we obtain
\begin{align*}
3 d_t&=3 C^*\cdot C^*_{1,0}+\sum_{p\in \mc{E}_*\cap C^*_{1,0}}C^*\cdot E_p,\\
 &=3 C^*\cdot E_{1,0}+3 C^*\cdot E_{0,1}+3C^*\cdot E_{1,1}+3 C^* \cdot E_{1,-\tau_0}.
\end{align*}
Therefore
\[
d_t=3 E_{\alpha_1,\beta_1}\cdot E_{1,0}+3 E_{\alpha_1,\beta_1}\cdot E_{0,1}+3E_{\alpha_1,\beta_1}\cdot E_{1,1}+
3 E_{\alpha_1,\beta_1}\cdot E_{1,-\tau_0}.
\]

Denoting $N(\alpha):=N_{\Z[\tau_0]}(\alpha)$, $\alpha\in \Z(\tau_0)$, by the example~\ref{ex1} we have
\[
3 d_t=3N(-\beta_1)+3N(\alpha_1)+3N(\alpha_1-\beta_1)+3N(-\alpha_1 \tau_0-\beta_1),
\]
Hence,
\begin{equation}\label{calgrado}
d_t=N(\beta_1)+N(\alpha_1)+N(\beta_1-\alpha_1)+N(\beta_1+\tau_0\alpha_1),
\end{equation}
where $\ds \Lambda(t)=\alpha=\frac{\alpha_1}{\beta_1}$.  
\begin{rema}
In~\eqref{calgrado} if $\alpha=a+\tau_0 b$ and $\beta=c+\tau_0 d$ then:
\[
d_t=d_t(a,b,c,d)=3( a^2 -  a b +  b^2 -  a c +  c^2 +  a d -  b d -  c d
+  d^2).
\]
In particular, $d_t$ is a multiple of 3.
\end{rema}

\subsection{The growth of the pencil $\mathcal{P}_4$}
In~\cite{Vit1}, Pereira defines the counting function $\pi_C$ of an algebraic $C$ curve included in $\fol(2,d)$, 
the space of foliations in $\P^2$ of degree $d$. In this case, if $\Pc=\{\F_{\alpha}\}_{\alpha\in\ovl{\ce}}$ is a line 
 in $\fol(2,d)$, that is a pencil of foliations in $\P^2$, then given $n\in\N$, we have
$\pi_\Pc(n)=\# E_n$, where
$$
E_n=\{\alpha\in\ovl{\ce}\::\:\F_{\alpha}\text{ have a first integral of degree at most }n\}
$$
is an algebraic set of $\ovl{\ce}$.

Also, in such paper the author observes the importance of study the function
$\pi_{\Pc}$ 
and shows 
the following example (cf.~\cite[Example 3]{Vit1}).
\begin{exem}\label{ejfacil}
Let $\Pc=\{\F_{\alpha}\}_{\alpha\in\ovl{\ce}}$ a pencil in $\P^2$, where $\F_{\alpha}$ is given by
$$
\alpha xdy-ydx.
$$
In this case, given $\alpha\in\ovl{\ce}$,
$$
\alpha\in I_p(\Pc)\setminus \{\infty\}\sii\alpha\in\Q.
$$
Thus, suppose that $\alpha=\frac{p}{q}$, $p\in\Z$, $q\in\N$ and $(p,q)=1$. Let
$f_{p,q}$ be the first integral of $\F_{\alpha}$ of degree $d_{p,q}$, then
\[
d_{p,q}=\begin{cases}
         \max\{p,q\},&\text{if }p\geq 0,\\
         |p|+q,&\text{if }p<0,
        \end{cases}
\]
Then, by doing simple calculations,
$$
\pi_{\Pc}(n)=
            2+3\sum_{j=1}^n\varphi(j),
$$
where $\varphi$ is the  Euler totient function. Now, since
$$
\sum_{j=1}^n\varphi(j)=\frac{3n^2}{\pi^2}+O
\Big(n\ln(n)^{2/3}\ln\big(\ln(n)\big)^{4/3}\Big),
$$
 (cf.~\cite[p. 178]{walf}), we have
$\ds\lim_{n\to\infty}\frac{\pi_{\Pc}(n)}{n^2}=\frac{3}{\pi^2}$.
\end{exem}

Now, we will estimate $\pi_{\Pc^4}(n)$, for $n\in \N$, and see that the counting
function $\pi_{\Pc^4}$ has the same behavior as in Example~\ref{ejfacil}.
\begin{coro}
$$
\ds \pi_{\Pc_4}(n)=O(n^2).
$$
\end{coro}
\begin{proof}
In fact, in this case
$$
t\in I_p(\Pc_4)\sii \Lambda^{-1}(t)=\alpha\in\Q(\tau_0)\cup \{\infty\},
$$
where $\Lambda(\alpha)=(\tau_0^2-1)\alpha+1$. Suppose that $\alpha=\frac{\alpha_1}{\beta_1}$, $\alpha_1,\beta_1\in\Z[\tau_0]$. 
Then
$$
\pi_ {\Pc_4}(n)=\#\Big\{\big(\alpha_1,\beta_1 \big)\in
(\Z[\tau_0]\times\Z[\tau_0])\setminus\{0\}\::\:(\alpha_1,\beta_1)=1,\ d_t\leq n
\Big\},
$$
where
$
d_t=N(\beta_1)+N(\alpha_1)+N(\beta_1-\alpha_1)+N(\beta_1+\tau_0\alpha_1).
$
Let
$$
\mc{E}_n= \Big\{\big(\alpha_1,\beta_1\big) \in (\Z[\tau_0]\times
I_p(\Pc))\setminus\{0\}\::\:t=\frac{\alpha_1}{\beta_1}, (\alpha_1,\beta_1)=1, \
N(\alpha_1) \leq n, N(\beta_1) \leq n \Big\},
$$
then
$$
\pi_{\Pc_4}(n)\leq \mc{E}_n, \quad \forall \  n \in \N.
$$
Let $H(n)=\big\{I \ \mbox{ideal in}  \ \Z[\tau_0] :N_{\Z[\tau_0]}(I) \leq n\big\}$ then by
\cite{ColJohn} we have
\begin{enumerate}
\item $H(n)=c n+O(n^{1/2})$, where $c$ is a constant.
\item\label{itemlim} $
\ds\lim_{n\to\infty}\frac{\mc{E}(n)}{H(n)^2} \leq \frac{1}{\zeta_{\Q(\tau_0)}(2)}
$, where $\zeta_{\Q(\tau_0)}$ is the Dedekind Zeta Function of $\Q(\tau_0)$ (see \S 2 ). 
\end{enumerate}
Therefore by the second item  we obtain
$$
\ds\lim_{n\to\infty}\frac{\pi_{\Pc_4}(n)}{H(n)^2} \leq \frac{1}{\zeta_{\Q(\tau_0)}(2)}.
$$
In particular $\ds \pi_{\Pc_4}(n)=O(n^2)$.
\end{proof}

\bibliographystyle{abbrv}
\bibliography{refer}
\end{document}